\documentclass[12pt]{article}
\usepackage{amsmath}
\usepackage{amssymb}
\usepackage{amsthm}
\usepackage{mathabx}
\usepackage{caption}
\usepackage[usenames]{color}
\usepackage{amscd}
\usepackage{dsfont}
\usepackage{indentfirst}

\usepackage[colorlinks=true,linkcolor=blue,filecolor=red,
citecolor=webgreen]{hyperref}
\definecolor{webgreen}{rgb}{0,.5,0}

\def\C{{\mathds{C}}}

\def\N{{\mathds{N}}}
\def\Z{{\mathds{Z}}}
\def\1{{\bf 1}}

\def\lcm{\operatorname{lcm}}
\def\gcud{\operatorname{gcud}}
\def\SL_2{\operatorname{SL_2}}

\newtheorem{theorem}{Theorem}

\newtheorem{lemma}[theorem]{Lemma}
\newtheorem{corollary}[theorem]{Corollary}

\begin{document}
	
	\title{{\bf On the asymptotic density of the ordered pairs $(a,b)$ of positive integers such that
			$\gcd(ab,a+b)=\gcd(a,b)$}}
	\author{Steve Fan and L\'aszl\'o T\'oth}
	\date{}
	\maketitle
	
	\begin{abstract} Consider the arithmetic function of two variables $f(a,b)= \gcd(ab,a+b)/\gcd(a,b)$, investigated in a recent preprint.
		We deduce asymptotic formulas for sums of the form $\sum_{a,b\le x} h(f(a,b))$, where $h$ belongs to a certain class of arithmetic functions. In particular, we obtain an asymptotic formula for the number of ordered pairs $(a,b)\in \N^2$ such that $a,b\le x$ and $f(a,b)=m$, where $m\in \N$ is fixed.
		This shows that in the case $m=1$ the corresponding density is the quadratic class number constant 
		$C= \prod_p (1-1/(p^2(p+1))) \doteq 0.881513$. 
	\end{abstract}
	
	
	
	\section{Introduction}
	
	It is well known that the probability that two random positive integers are relatively prime is $6/\pi^2\doteq 0.607927$.
	More precisely, this value is the asymptotic density of the ordered pairs $(a,b)$ such that
	$a,b\in \N:=\{1,2,\ldots \}$ are relatively prime. That is, 
	\begin{align*}
		\lim_{x \to \infty} \frac1{x^2} \sum_{\substack{a,b\le x\\ \gcd(a,b)=1}} 1 = \frac{6}{\pi^2}.
	\end{align*}
	
	In fact, one has the asymptotic formula 
	\begin{align}  \label{asympt_gcd}
		\sum_{\substack{a,b\le x\\ \gcd(a,b)=1}} 1 = \frac{6}{\pi^2}x^2 +O(x(\log x)^{2/3} (\log \log x)^{1/3}),
	\end{align}
	where the error term is inherited from the best known estimate for the mean value of Euler's totient function due to Liu 
	\cite{Liu2016}, improving the earlier known error by Walfisz \cite{Wal1963}.
	
	There are many other asymptotic formulas in the literature concerning the greatest common divisor (gcd) and least common multiple (lcm) in two or more variables.
	See, e.g., \cite{BMR2019, HT2021, HLT2020, HilTot2016, IMR2022, Kim2023, Tot2010} and their references.
	For example, formula
	\begin{equation} \label{gcd_a_b}
		\sum_{a,b\le x} \gcd(a,b)= \frac{x^2}{\zeta(2)}\left(\log x+ 2\gamma
		-\frac1{2}-\frac{\zeta(2)}{2}- \frac{\zeta'(2)}{\zeta(2)} \right) +
		O\left(x^{1+\theta+\varepsilon}\right)
	\end{equation}
	holds for every $\varepsilon>0$, where $\zeta$ is the Riemann zeta function, $\gamma$ is Euler's constant, and $\theta$ denotes 
	the exponent appearing in Dirichlet's divisor problem. Also,
	\begin{equation} \label{gcd_recipr_a_b}
		\sum_{a,b\le x} \frac1{\gcd(a,b)} = \frac{\zeta(3)}{\zeta(2)} x^2 + O\left(x(\log x)^{2/3}(\log \log x)^{1/3} \right),
	\end{equation}
	with the same error term as in \eqref{asympt_gcd}.
	
	Recently, Thang Pang Ern and Malcolm Tan Jun Xi \cite{TT2025} investigated the arithmetic function
	of two variables
	\begin{equation} \label{def_f}
		f(a,b)= \frac{\gcd(ab,a+b)}{\gcd(a,b)},
	\end{equation}
	having connections to the orbits of certain actions of the group $\SL_2(\Z)$.
	
	Let us denote $a=\delta i$, $b=\delta j$, where $\gcd(i,j)=1$. Then $\gcd(a,b)=\delta$ and we have $f(a,b)= \gcd(\delta^2ij,\delta(i+j))/\delta = \gcd(\delta ij, i+j)$, that is, 
	\begin{align} \label{form_f}
		f(a,b)= \gcd(\delta, i+j),
	\end{align} 
	since $\gcd(i,j)=1$ implies that $\gcd(ij,i+j)=1$. It follows that the function $f$ takes only positive integer values. Moreover, $f(a,b)\mid \gcd(a,b)$ for every $a,b\in \N$. 
	
	Hence, if $\gcd(a,b)=1$, then $f(a,b)=1$. The converse is not true. For example, $f(p^s,p^t)=1$ for all odd primes $p$
	and all $s,t\in \N$. Also, $f(2p,3p)=1$ for all primes $p\ne 5$.
	
	Therefore, a natural question is to find the asymptotic density of the ordered pairs $(a,b)\in \N^2$ such that $f(a,b)=1$.
	The authors \cite{TT2025} obtained that this asymptotic density is
	\begin{align} \label{def_C}
		C= \prod_p \left(1-\frac1{p^2(p+1)} \right) \doteq 0.881513,
	\end{align}
	the quadratic class number constant (see sequence A065465 of \cite{OEIS}).
	However, their proof is heuristic, using that for primes $p$ the events $p\mid i$ and $p\mid j$ are independent.
	
	The purpose of the present paper is to prove by exact elementary arguments that the corresponding 
	asymptotic density is indeed $C$. This is a consequence of the following asymptotic formula with a remainder term:
	
	\begin{equation} \label{sum_1}
		\sum_{\substack{a,b\le x \\ f(a,b)=1}} 1 = C x^2 + O\left(x(\log x)^3 \right), 
	\end{equation}
	where the constant $C$ is given by \eqref{def_C}. Formula \eqref{sum_1} is the analogue of \eqref{asympt_gcd}.
	
	We also deduce a similar formula for the number of ordered pairs $(a,b)\in \N^2$ such that $a,b\le x$ and $f(a,b)=m$, where $m\in \N$ is fixed, and more generally for sums of the form $\sum_{a,b\le x} h(f(a,b))$, where $h$ belongs to a certain class of arithmetic functions. In particular,
	we derive the following formula, concerning the mean value of the function $f(a,b)$:
	\begin{equation} \label{asympt_f}
		\sum_{a,b\le x} f(a,b) = K x^2 + O(x^{3/2}\log x),
	\end{equation}where
	\begin{equation} \label{K}
		K= \prod_p \left(1+\frac1{(p+1)^2}\right) \doteq 1.266558,
	\end{equation}
	(see sequence A065486 in \cite{OEIS}). Note that formula \eqref{asympt_f} is the analogue of \eqref{gcd_a_b}. 
	
	We remark that the gcd function $(a,b)\mapsto \gcd(a,b)$ and the lcm function
	$(a,b)\mapsto \lcm(a,b)$ are multiplicative, viewed as functions of two variables. See Section \ref{Section_Further_remarks} for this notion. The function $f$ defined by \eqref{def_f} is not multiplicative and it is more challenging to study its properties.
	
	Our main results are included in Section \ref{Section_Results}. The proofs, including the proofs of formulas \eqref{sum_1} and \eqref{asympt_f}, are given in Section \ref{Section_Proofs}. Some remarks and directions for further research are presented in Section \ref{Section_Further_remarks}. 
	
	We use the following notation: $\mu$ is the M\"obius function; $\tau(n)=\sum_{d\mid n} 1$ is the divisor function; $\alpha(n)$ is the 
	number of squarefull divisors of $n$; $\sigma(n)$ is the sum-of-divisors function; $\phi(n)=n\prod_{p\mid n} \left(1-1/p\right)$ is Euler's function; $\psi(n)=n\prod_{p\mid n} \left(1+1/p\right)$ denotes the Dedekind function;
	$\omega(n)$ and $\Omega(n)$ stand for the number of distinct prime divisors of $n$ and the number of prime power divisors of $n$, respectively; 
	$*$ denotes the convolution of arithmetic functions; $\sum_p$ and $\prod_p$ represent sums and products over the primes, respectively.
	
	\section{Results} \label{Section_Results}
	
	In the asymptotic formulas we assume that $x>1$ is a real number. We have the following general result.
	
	\begin{theorem} \label{Th_h} Let $h:\N \to \C$ be an arithmetic function such that $(\mu * h)(n)\ll n^{\beta}$, where
		$\beta<2$ is a fixed real number. Then
		\begin{equation*} 
			\sum_{a,b\le x} h(f(a,b))  =  x^2 \sum_{n=1}^{\infty} \frac{(\mu*h)(n)}{n^2\psi(n)} + R_{\beta}(x), 
		\end{equation*}
		where
		\begin{equation*} 
			R_{\beta}(x) \ll \begin{cases}   
				x^{1+\beta/2} \log x, & \text{ if $0<\beta<2$}, \\
				x(\log x)^3, & \text{ if $\beta=0$ and $\sum_{n=1}^{\infty} |(\mu*h)(n)|/n = \infty$}, \\ 
				x (\log x)^2, & \text{ if $\beta=0$ and $\sum_{n=1}^{\infty} |(\mu*h)(n)|/n < \infty$ or $\beta<0$}.
			\end{cases}
		\end{equation*}
	\end{theorem}
	
	If $h(m)=1$ and $h(n)=0$ for $n\ne m$, then we have the following result.
	
	\begin{corollary} \label{Cor_m} Let $m\in \N$ be fixed. Then
		\begin{equation} \label{sum_density_m}
			\sum_{\substack{a,b\le x\\ f(a,b)=m}} 1  =  C_m x^2 + O\left(x(\log x)^3 \right), 
		\end{equation} 
		where 
		\begin{equation*} 
			C_m= \frac{C}{m^3} \prod_{p\mid m} \left(1-\frac{p^2}{p^3+p^2-1}\right),
		\end{equation*} 
		$C$ being defined by \eqref{def_C}.
	\end{corollary}
	
	In the case $m=1$ formula \eqref{sum_density_m} recovers \eqref{sum_1} and $C_1=C$.  
	Some more numerical values of $C_m$, representing the asymptotic density of the ordered pairs $(a,b)\in \N^2$ such that $f(a,b)=m$ are
	\[
	C_2=\frac{7}{88}C \doteq 0.070120, \ C_3=\frac{26}{945}C \doteq 0.024253, \ C_4=\frac{7}{704}C \doteq 0.008765,
	\]
	\[
	C_5=\frac{124}{18625}C \doteq 0.005868, \ C_6=\frac{91}{41580}C \doteq 0.001929.
	\]
	
	For some other choices of $h$ we can deduce similar formulas. We present the following examples.
	
	\begin{corollary} {\rm ($h(n)=n^t$)} \label{Cor_n_beta} 
		If $t<2$ \textup{($t\ne 0$)} is a real number, then
		\begin{equation*} 
			\sum_{a,b\le x} (f(a,b))^t = x^2 \prod_p \left(1+ \frac{p(p^t-1)}{(p+1)(p^3-p^t)}\right) + E_t(x),
		\end{equation*} 
		where
		\begin{equation*} 
			E_t(x) \ll \begin{cases}   
				x^{1+t/2} \log x, & \text{ if $0<t<2$}, \\
				x(\log x)^3, & \text{ if $t<0$}.
			\end{cases}
		\end{equation*}
	\end{corollary}
	
	In the case $t=1$ we recover formula \eqref{asympt_f}.
	
	\begin{corollary} {\rm ($h(n)=n^{-1}$)} \label{Cor_1_per_n} 
		\begin{equation} \label{sum_1_per_n}
			\sum_{a,b\le x} \frac1{f(a,b)} = L x^2  + O(x (\log x)^3),
		\end{equation} 
		where
		\begin{equation} \label{def_L}
			L=  \prod_p \left(1- \frac{p}{(p^2+1)(p+1)^2}\right).
		\end{equation} 
	\end{corollary}
	
	Compare formula \eqref{sum_1_per_n} to \eqref{gcd_recipr_a_b}. 
	The following result concerns the asymptotic density of the ordered pairs $(a,b)\in \N^2$
	such that $f(a,b)$ is squarefree. 
	
	\begin{corollary} {\rm ($h=\mu^2$)} \label{Cor_mu_2}
		\begin{equation*} 
			\sum_{a,b\le x} \mu^2(f(a,b)) = C_{\mu^2} x^2 + O\left(x(\log x)^2 \right),
		\end{equation*} 
		where
		\begin{equation*} 
			C_{\mu^2} = \prod_p \left(1-\frac1{p^5(p+1)}\right) \doteq 0.988504.
		\end{equation*} 
	\end{corollary}
	
	For the constant $C_{\mu^2}$ see \cite{OEIS}, sequence A065468.
	
	\begin{corollary} {\rm ($h=\mu$)} \label{Cor_mu} For every $\varepsilon>0$,
		\begin{equation*} 
			\sum_{a,b\le x} \mu(f(a,b)) = x^2 \prod_p \left(1-\frac{2p^3-1}{p^5(p+1)}\right)  + O\left(x^{1+\varepsilon}\right).
		\end{equation*} 
	\end{corollary}
	
	\begin{corollary} {\rm ($h=\tau$)} \label{Cor_tau}
		\begin{equation*} 
			\sum_{a,b\le x} \tau(f(a,b)) =  x^2 \prod_p \left(1+\frac{p}{(p+1)(p^3-1)}\right)  + O\left(x(\log x)^3 \right).
		\end{equation*} 
	\end{corollary}
	
	\begin{corollary} {\rm ($h=\alpha$)} \label{Cor_alpha}
		\begin{equation*} 
			\sum_{a,b\le x} \alpha(f(a,b)) = x^2 \prod_p \left(1+\frac1{p^2(p+1)(p^3-1)}\right) + O\left(x(\log x)^2 \right).
		\end{equation*} 
	\end{corollary}
	
	\begin{corollary} {\rm ($h=\omega$)} \label{Cor_omega}
		\begin{equation*}
			\sum_{a,b\le x} \omega(f(a,b)) = C_{\omega} x^2  + O\left(x(\log x)^3 \right),
		\end{equation*} 
		where
		\begin{equation*} 
			C_{\omega} = \sum_p \frac1{p^2(p+1)} \doteq 0.122017.
		\end{equation*} 
	\end{corollary}
	
	\begin{corollary} {\rm ($h=\Omega$)} \label{Cor_Omega}
		\begin{equation*} 
			\sum_{a,b\le x} \Omega(f(a,b)) = C_{\Omega} x^2  + O\left(x(\log x)^3 \right),
		\end{equation*} 
		where
		\begin{equation*} 
			C_{\Omega} = \sum_p \frac{p}{(p+1)(p^3-1)} \doteq 0.135052.
		\end{equation*} 
	\end{corollary}
	
	For the constants $C_{\omega}$ and $C_{\Omega}$ see sequences A382562 and A397427 of \cite{OEIS}.

	\begin{corollary} {\rm ($h=\sigma$)} \label{Cor_sigma} 
		\begin{equation*} 
			\sum_{a,b\le x} \sigma(f(a,b)) = C_{\sigma} x^2  + O(x^{3/2} \log x),
		\end{equation*} 
		where 
		\begin{equation*}
			C_{\sigma} = \zeta(2)C= \prod_p \left(1+\frac{p}{(p-1)(p+1)^2}\right) \doteq 1.450032,   
		\end{equation*}
		$C$ being defined by \eqref{def_C}. 
	\end{corollary}
	
	For the constant $C_{\sigma}$ see \cite{OEIS}, sequence A335762.
	
	\begin{corollary} {\rm ($h=\phi$)} \label{Cor_phi} 
		\begin{equation*}
			\sum_{a,b\le x} \phi(f(a,b)) = C_{\phi} x^2 + O(x^{3/2} \log x),
		\end{equation*} 
		where
		\begin{equation*}
			C_{\phi} = \prod_p \left(1+\frac{p^3-p^2-p-1}{p^3(p+1)^2}\right).   
		\end{equation*}
	\end{corollary}
	
	Formulas for the sums $\sum_{a,b\le x} \sigma(\gcd(a,b))$ and $\sum_{a,b\le x} \phi(\gcd(a,b))$
	are known in the literature. See Cohen \cite{Coh1960, Coh1962}.
	
	\begin{corollary} {\rm ($h=\psi$)} \label{Cor_psi} 
		\begin{equation*} 
			\sum_{a,b\le x} \psi(f(a,b)) = C_{\psi} x^2 + O(x^{3/2} \log x),
		\end{equation*} 
		where
		\begin{equation*}
			C_{\psi} = \prod_p \left(1+\frac{p^2+1}{p^3(p+1)}\right).   
		\end{equation*}
	\end{corollary}

	\section{Proofs} \label{Section_Proofs}
	
	We will only use familiar elementary estimates.    
	The following two lemmas are known and can easily be deduced from the familiar properties $\sum_{d\mid n} \mu(d)= \lfloor 1/n \rfloor$ and $\phi(n)= n\sum_{d\mid n} \mu(d)/d$ ($n\in \N$).
	
	\begin{lemma} \label{Lemma_1} Let $a$ be an integer and $m,q$ be positive integers such that $\gcd(m,q) =1$.
		Then 
		\begin{equation*}
			\sum_{\substack{n\le x\\ \gcd(n,m)=1 \\ n \equiv a \textup{ (mod $q$)}}} 1= \frac{\phi(m)x}{mq} + O(2^{\omega(m)}),
		\end{equation*}
		uniformly for $a,m,q$.
	\end{lemma}
	
	\begin{lemma} \label{Lemma_2} Let $m$ be a positive integer. Then 
		\begin{equation*}
			\sum_{\substack{n\le x \\  \gcd(n,m)=1}} \frac{\phi(n)}{n} = \frac{mx}{\zeta(2)\psi(m)} + O(2^{\omega(m)}\log x),
		\end{equation*}
		uniformly for $m$.
	\end{lemma}
	
	Now we prove the key lemma of our treatment.
	
	\begin{lemma} \label{Lemma_V} Let $d$ be a positive integer. Then 
		\begin{equation*}
			V_d(x):= \sum_{\substack{i,j\le x \\  \gcd(i,j)=1\\ d\mid (i+j)}} 1 = \frac{x^2}{\zeta(2)\psi(d)} + O(x\log x),
		\end{equation*}
		uniformly for $d\in \N$.
	\end{lemma}
	
	\begin{proof} If $ d\mid (i+j)$ and $\gcd(i,j)=1$, then $\gcd(i,d)=\gcd(j,d)=1$.  Using Lemmas \ref{Lemma_1} and \ref{Lemma_2}, 
		\[
		V_d(x)= \sum_{\substack{j\le x \\  \gcd(j,d)=1}} 1 \sum_{\substack{i\le x\\ \gcd(i,j)=1\\ i\equiv -j \text{ (mod $d$)}}} 1 
		\]
		\[
		= \sum_{\substack{j\le x \\  \gcd(j,d)=1}} \left(\frac{\phi(j)x}{jd} + O(2^{\omega(j)}) \right) =
		\frac{x}{d} \sum_{\substack{j\le x \\  \gcd(j,d)=1}} \frac{\phi(j)}{j} +  O\left(\sum_{j\le x} 2^{\omega(j)} \right) 
		\]
		\[
		= \frac{x}{d} \left( \frac{dx}{\zeta(2)\psi(d)}  +  O\left(2^{\omega(d)}\log x \right)\right) + O(x\log x)
		\]
		\[
		= \frac{x^2}{\zeta(2)\psi(d)}  +  O\left(\frac{2^{\omega(d)}}{d} x \log x \right) + O(x\log x)
		\]
		\[
		= \frac{x^2}{\zeta(2)\psi(d)}  +  O(x\log x),
		\]
		where $2^{\omega(d)}/d\le 1$ for every $d\in \N$.
	\end{proof}
	
	\begin{lemma} \label{Lemma_h} Let $h$ be an arbitrary arithmetic function. Then 
		\begin{equation*}
			S(x):= \sum_{a,b\le x} h(f(a,b)) = x^2 \sum_{d\le (2x)^{1/2}} \frac{(\mu*h)(d)}{d^2\psi(d)} + S_1(x) +S_2(x),
		\end{equation*}
		where
		\begin{align*}
			S_1(x) & \ll x \sum_{d\le (2x)^{1/2}} \frac{|(\mu*h)(d)|}{\psi(d)}, \\
			S_2(x) & \ll x \sum_{d^2s\le 2x} \frac{|(\mu*h)(d)|}{ds}  \log \frac{3x}{ds}.
		\end{align*}
	\end{lemma}
	
	\begin{proof}[Proof of Lemma {\rm \ref{Lemma_h}}]
		We have, according to \eqref{form_f} and using that $h(n)=\sum_{d\mid n} (\mu*h)(d)$ ($n\in \N$),
		\begin{equation*}
			S(x) = \sum_{\substack{\delta i, \delta j\le x \\ \gcd(i,j)=1}} h(\gcd(\delta,i+j)) =
			\sum_{\substack{\delta i, \delta j\le x \\ \gcd(i,j)=1}} \sum_{d\mid \gcd(\delta,i+j)} (\mu*h)(d) 
		\end{equation*}
		\begin{equation*}
			= \sum_{\substack{\delta i, \delta j\le x \\ \gcd(i,j)=1}} \sum_{\substack{d\mid \delta\\ d\mid (i+j)}} (\mu*h)(d)
			= \sum_{\substack{dsi, dsj\le x \\ \gcd(i,j)=1\\  d\mid (i+j)}} (\mu*h)(d).
		\end{equation*}
		
		Here $d\mid i+j$ and $i,j\le x/(ds)$ imply that $d\le 2\max(i,j)\le 2x/(ds)$, that is, $d^2s\le 2x$. We obtain 
		\begin{equation*}
			S(x)= \sum_{d\le (2x)^{1/2}} (\mu*h)(d) \sum_{s\le 2x/d^2}  \sum_{\substack{i, j\le x/(ds)\\ \gcd(i,j)=1\\ d\mid (i+j)}} 1
		\end{equation*}
		\begin{equation*}
			=\sum_{d\le (2x)^{1/2}} (\mu*h)(d) \sum_{s\le 2x/d^2} V_d(x/(ds)).
		\end{equation*}
		
		Now using Lemma \ref{Lemma_V} we deduce for $x>1$,
		\begin{equation*}
			S(x)=  \sum_{d\le (2x)^{1/2}} (\mu*h)(d) \sum_{s\le 2x/d^2} \left( \frac{x^2}{\zeta(2)d^2s^2\psi(d)}+ O\left( \frac{x}{ds}\log \frac{3x}{ds}\right) \right)
		\end{equation*}
		\begin{equation*}
			=  \frac{x^2}{\zeta(2)} \sum_{d\le (2x)^{1/2}} \frac{(\mu*h)(d)}{d^2\psi(d)} \sum_{s\le 2x/d^2} \frac1{s^2} 
			+ \sum_{d\le (2x)^{1/2}} (\mu*h)(d) \sum_{s\le 2x/d^2} O\left(\frac{x}{ds}\log \frac{3x}{ds}\right),
		\end{equation*}
		where we have $sd\le sd^2\le 2x < 3x$, hence $3x/(ds)>1$. Since $\sum_{s\le 2x/d^2} \frac1{s^2}= \zeta(2) + O(d^2/x)$, we obtain 
		the given estimate, which is valid for every function $h$.
	\end{proof}
	
	\begin{proof}[Proof of Theorem {\rm \ref{Th_h}}]
		Assume that $(\mu*h)(n)\ll n^{\beta}$, where $\beta<2$. Then by Lemma \ref{Lemma_h} we obtain
		\begin{equation*}
			S(x) =  x^2 \sum_{d=1}^{\infty} \frac{(\mu*h)(d)}{d^2\psi(d)} + T(x)+ S_1(x) +S_2(x),
		\end{equation*}
		say, where the series is absolutely convergent and  
		\begin{equation*}
			T(x): = - x^2 \sum_{d>(2x)^{1/2}} \frac{(\mu*h)(d)}{d^2\psi(d)} \ll x^2 \sum_{d>(2x)^{1/2}} \frac{|(\mu*h)(d)|}{d^2\psi(d)} \ll 
			x^2 \sum_{d>(2x)^{1/2}} \frac1{d^{3-\beta}} \ll x^{1+\beta/2}
		\end{equation*}
		for every $\beta<2$. Also,
		\begin{equation*}
			S_1(x)\ll x \sum_{d\le (2x)^{1/2}} \frac{|(\mu*h)(d)|}{\psi(d)} \ll x \sum_{d\le (2x)^{1/2}} \frac1{d^{1-\beta}}, 
		\end{equation*}
		giving $S_1(x)\ll x^{1+\beta/2}$ if $0<\beta<2$, $S_1(x)\ll x\log x$ if $\beta=0$ and $S_1(x)\ll x$ if $\beta<0$. Furthermore,
		\begin{equation*}
			S_2(x)\ll x \sum_{d^2s\le 2x} \frac{|(\mu*h)(d)|}{ds} \log \frac{3x}{ds} \ll 
			x (\log x) \sum_{s\le 2x} \frac1{s} \sum_{d\le (2x/s)^{1/2}} \frac1{d^{1-\beta}}.
		\end{equation*}
		
		If $0<\beta<2$, then $-1<1-\beta <1$ and we deduce
		\begin{equation*}
			S_2(x) \ll x (\log x) \sum_{s\le 2x} \frac1{s} \left(\frac{x}{s}\right)^{\beta/2} \ll 
			x^{1+\beta/2}(\log x) \sum_{s\le 2x} \frac1{s^{1+\beta/2}}\ll x^{1+\beta/2} \log x.
		\end{equation*}
		
		If $\beta=0$, then
		\begin{equation*}
			S_2(x) \ll x (\log x) \sum_{s\le 2x} \frac1{s} \sum_{d\le (2x/s)^{1/2}} \frac1{d}\ll 
			x(\log x)^3.
		\end{equation*}
		
		If $\beta=0$ and the series  $\sum_{n=1}^{\infty} (\mu*h)(n)/n$ is absolutely convergent, then 
		\begin{equation*}
			S_2(x)\ll x (\log x) \sum_{s\le 2x} \frac1{s} \sum_{d\le (2x/s)^{1/2}} \frac{|(\mu*h)(d)|}{d}
			\ll x(\log x) \sum_{s\le 2x} \frac1{s}\ll x(\log x)^2.
		\end{equation*}
		
		Finally, if $\beta<0$, then $1-\beta>1$ and we have    
		\begin{equation*}
			S_2(x) \ll x(\log x) \sum_{s\le 2x} \frac1{s} \ll x(\log x)^2,
		\end{equation*}
		completing the proof.
	\end{proof}
	
	\begin{proof}[Proof of Corollary {\rm \ref{Cor_m}}] Apply Theorem \ref{Th_h} for the function $h=h_m$, where
		$h_m(m)=1$, $h_m(n)=0$ for $n\ne m$. Then the function
		\begin{equation*}
			(\mu*h_m)(n)= \sum_{d\mid n} h_m(d)\mu(n/d) = \begin{cases} \mu(n/m), & \text{ if $m\mid n$,}\\ 0, & \text{ otherwise,}
			\end{cases}
		\end{equation*}
		is bounded (case $\beta=0$), and the constant of the main term is
		\begin{equation*}
			C_m:= \sum_{n=1}^{\infty} \frac{(\mu*h_m)(n)}{n^2\psi(n)}= \sum_{\substack{n=1\\ m\mid n}}^{\infty} \frac{\mu(n/m)}{n^2\psi(n)}
			= \sum_{k=1}^{\infty} \frac{\mu(k)}{k^2m^2\psi(km)},
		\end{equation*}
		where $\psi(km)=\psi(k)\psi(m)\gcd(k,m)/\psi(\gcd(k,m))$ ($k,m\in \N$), 
		an identity analogous to the corresponding identity for Euler's totient function. Also, note that if a function $F$ is 
		multiplicative, then for fixed $m$ the function $k\mapsto F(\gcd(k,m))$ is also multiplicative. We deduce by expanding into an Euler product that 
		\begin{equation*}
			C_m= \frac1{m^2\psi(m)} \sum_{k=1}^{\infty} \frac{\mu(k)\psi(\gcd(k,m))}{k^2\psi(k)\gcd(k,m)}
			=\frac1{m^2\psi(m)}  \prod_p \left(1+\frac{\mu(p)\psi(\gcd(p,m))}{p^2\psi(p)\gcd(p,m)} \right)
		\end{equation*}
		\begin{equation*}
			=\frac1{m^2\psi(m)} \prod_{p\, \nmid \, m} \left(1- \frac1{p^2(p+1)} \right) \prod_{p\mid m} \left(1- \frac1{p^3}\right)
		\end{equation*}
		\begin{equation*}
			=\frac{C}{m^2\psi(m)} \prod_{p\mid m} \left(1- \frac1{p^3}\right) \left(1- \frac1{p^2(p+1)} \right)^{-1} = \frac{C}{m^3} \prod_{p\mid m} \left(1-\frac{p^2}{p^3+p^2-1}\right),  
		\end{equation*}
		where $C$ is the constant defined by \eqref{def_C}.
	\end{proof}
	
	\begin{proof}[Proof of Corollary {\rm \ref{Cor_n_beta}}]
		In the case $h(n)=n^t$ we have $(\mu*h)(n)= n^t \prod_{p\mid n} \left(1-\frac1{p^t}\right)=\phi_t(n)$, the generalized Euler function, which is multiplicative. Here
		$0< \phi_t(n)\le n^t$ ($n\in\N$) if $0<t<2$, and $|\phi_t(n)|\le 1$ ($n\in\N$) if $t<0$. Hence Theorem \ref{Th_h}
		applies with the error term corresponding to $\beta=t$ if $0<t<2$ and $\beta=0$ in the case $t<0$.
		The constant of the main term is
		\begin{equation*}
			\sum_{n=1}^{\infty} \frac{\phi_t(n)}{n^2\psi(n)}= \prod_p \left(1+ \sum_{a=1}^{\infty} \frac{\phi_t(p^a)}{p^{2a}\psi(p^a)}\right)= 
			\prod_p \left(1+ \frac{p(p^t-1)}{(p+1)(p^3-p^t)}\right). \qedhere
		\end{equation*}
	\end{proof}
	
	\begin{proof}[Proof of Corollary {\rm \ref{Cor_1_per_n}}]
		Apply Corollary \ref{Cor_n_beta} in the case $t=-1$.
	\end{proof}
	
	\begin{proof}[Proof of Corollary {\rm \ref{Cor_mu_2}}]
		Apply Theorem \ref{Th_h} for the function $h=\mu^2$, where $g=\mu*\mu^2$ is multiplicative, 
		$g(p^2)=-1$ and $g(p^a)=0$ for $a\in \N\setminus \{2\}$ (case $\beta=0)$. Also, the series
		\begin{equation*}
			\sum_{n=1}^{\infty} \frac{|g(n)|}{n}= \prod_p \left(1+\frac1{p^2}\right) 
		\end{equation*}
		converges. The constant of the main term is
		\begin{equation*}
			\sum_{n=1}^{\infty} \frac{g(n)}{n^2\psi(n)}= \prod_p \left(1-\frac1{p^5(p+1)}\right). \qedhere
		\end{equation*}
	\end{proof}
	
	\begin{proof}[Proof of Corollary {\rm \ref{Cor_mu}}]
		In the case of $h=\mu$ the function $F=\mu*\mu$ is multiplicative, 
		$F(p)=-2$, $F(p^2)=1$ and $F(p^a)=0$ for $a\ge 3$. Hence $|F(n)|\le 2^{\omega(n)}\ll n^{\varepsilon}$.  
		The constant of the main term is
		\begin{equation*}
			\sum_{n=1}^{\infty} \frac{F(n)}{n^2\psi(n)}= \prod_p \left(1-\frac{2p^3-1}{p^5(p+1)}\right). \qedhere
		\end{equation*}
	\end{proof}
	
	\begin{proof}[Proof of Corollary {\rm \ref{Cor_tau}}] Apply Theorem \ref{Th_h} for the function $h=\tau$, where
		$(\mu*\tau)(n)=1$ ($n\in \N$) is bounded. The related constant is
		\begin{equation*}
			\sum_{n=1}^{\infty} \frac1{n^2\psi(n)}=  \prod_p \left(1+\frac{p}{(p+1)(p^3-1)}\right). \qedhere
		\end{equation*}
	\end{proof}
	
	\begin{proof}[Proof of Corollary {\rm \ref{Cor_alpha}}] Here for
		$G=\mu*\alpha$ we have $G(p)=0$ and $G(p^a)=1$ for $a\ge 2$ (case $\beta=0$, again).  Also, the series
		\begin{equation*}
			\sum_{n=1}^{\infty} \frac{|G(n)|}{n}= \prod_p \left(1+\sum_{a=2}^{\infty} \frac1{p^a}\right)= \prod_p \left(1+\frac1{p(p-1)}\right) 
		\end{equation*}
		converges. The related constant is
		\begin{equation*}
			\sum_{n=1}^{\infty} \frac{G(n)}{n^2\psi(n)}=  \prod_p \left(1+\frac1{p^2(p+1)(p^3-1)}\right). \qedhere
		\end{equation*}
	\end{proof}
	
	\begin{proof}[Proof of Corollary {\rm \ref{Cor_omega}}] Apply Theorem \ref{Th_h} for the function $h=\omega$.
		Since $\omega(n)=\sum_{p\mid n} 1$ we deduce that $(\mu*\omega)(n)$ equals $1$ if $n$ is a prime and $0$ otherwise. 
		The related constant is
		\begin{equation*}
			\sum_{n=1}^{\infty} \frac{(\mu*\omega)(n)}{n^2\psi(n)}= \sum_p \frac1{p^2(p+1)}. \qedhere 
		\end{equation*}
	\end{proof}
	
	\begin{proof}[Proof of Corollary {\rm \ref{Cor_Omega}}] If $h=\Omega$, then
		$(\mu*\Omega)(n)$ equals $1$ if $n$ is a prime power and $0$ otherwise. The 
		corresponding constant is
		\begin{equation*}
			\sum_{n=1}^{\infty} \frac{(\mu*\Omega)(n)}{n^2\psi(n)}= \sum_p \sum_{a=1}^{\infty} \frac1{p^{2a}\psi(p^a)}=
			\sum_p \frac{p}{(p+1)(p^3-1)}. \qedhere
		\end{equation*}
	\end{proof}
	
	\begin{proof}[Proof of Corollary {\rm \ref{Cor_sigma}}] Here $h=\sigma$ and 
		$(\mu*\sigma)(n)=n$ for every $n\in \N$ (case $\beta=1$).  The 
		corresponding constant is
		\begin{equation*}
			\sum_{n=1}^{\infty} \frac1{n\psi(n)}= \prod_p \left(1+\frac{p}{(p-1)(p+1)^2} \right). \qedhere
		\end{equation*}
	\end{proof}
	
	\begin{proof}[Proof of Corollary {\rm \ref{Cor_phi}}] If $h=\phi$, then
		$(\mu*\phi)(p)=p-2$ for every prime $p$ and $(\mu*\phi)(p^a)= p^a(1-1/p)^2$ for every prime power $p^a$ with $a\ge 2$. 
		Also,
		\begin{equation*}
			\sum_{n=1}^{\infty} \frac{(\mu*\phi)(n)}{n^2\psi(n)}= \prod_p \left(1+\frac{p^3-p^2-p-1}{p^3(p+1)^2}\right).  \qedhere
		\end{equation*}
	\end{proof}
	
	\begin{proof}[Proof of Corollary {\rm \ref{Cor_psi}}] If $h=\psi$, then
		$(\mu*\psi)(p)=p$ for every prime $p$ and $(\mu*\psi)(p^a)= p^a(1-1/p^2)$ for every prime power $p^a$ with $a\ge 2$. 
		Here
		\begin{equation*}
			\sum_{n=1}^{\infty} \frac{(\mu*\psi)(n)}{n^2\psi(n)}= \prod_p \left(1+\frac{p^2+1}{p^3(p+1)}\right).  \qedhere
		\end{equation*}
	\end{proof}
	
	\section{Further remarks} \label{Section_Further_remarks}

	1. A nonzero arithmetic function of two variables $g:\N^2\to \C$ is called multiplicative if 
	\begin{equation*}
		g(m_1n_1,m_2n_2) = g(m_1,m_2) g(n_1,n_2)
	\end{equation*}
	holds for all $m_1,m_2,n_1,n_2\in \N$ such that $\gcd(m_1m_2,n_1n_2)=1$.
	If $g$ is multiplicative, then it is determined by the values
	$g(p^{\nu_1},p^{\nu_2})$, where $p$ is prime and
	$\nu_1,\nu_2\in \N \cup \{0\}$. More precisely, $g(1,1)=1$ and
	for all $n_1,n_2\in \N$,
	\begin{align*} 
		g(n_1,n_2)= \prod_p g(p^{\nu_p(n_1)},p^{\nu_p(n_2)}),
	\end{align*}
	where we use the notation $n=\prod_p p^{\nu_p(n)}$ for the prime power 
	factorization of $n\in \N$. See the survey \cite{Tot2014} for details on the concept
	of multiplicative functions of several variables.
	As mentioned in the Introduction,
	the gcd and lcm functions are multiplicative, but the function $f$ investigated in this paper is not multiplicative. 
	For example, $f(4,6)f(5,5)= 1\ne 5= f(20,30)$. 
	
	2. Another representation of the function $f$, communicated by Ramar\'e \cite{R}, is given by
	\begin{equation*}
		f(a,b)= \gcd(\lambda,a'+b'),
	\end{equation*}
	where $a=\lambda a'$, $b=\lambda b'$ and $\lambda = \gcud(a,b)$ is the greatest common unitary divisor of $a$ and $b$. We recall that a divisor $d$ of $n$ is a unitary divisor (block divisor) if $\gcd(d,n/d)=1$. Here $\gcud(a,b) \mid \gcd(a,b)$. It turns out that $f(a,b)\mid \gcud(a,b)$ for every $a,b\in \N$. 
	
	It is known, see the second author \cite[Th.\ 2.1]{Tot2001}, that 
	\begin{equation} \label{gcdu_1}
		\sum_{\substack{a,b\le x \\ \gcud(a,b)=1}} 1 = D x^2 + O\left(x(\log x)^2 \right), 
	\end{equation}
	where 
	\begin{equation*}
		D=\prod_p \left(1-\frac{p-1}{p^2(p+1)} \right)\doteq 0.807330,
	\end{equation*}
	see sequence A306071 in \cite{OEIS}. Compare formula \eqref{gcdu_1} to \eqref{asympt_gcd}. Also, 
	\begin{equation} \label{sum_gcud}
		\sum_{a,b\le x} \gcud(a,b) = Fx^2\log x +O(x^2),    
	\end{equation}
	where 
	\begin{equation*}
		F=\zeta(2) \prod_p \left(1- \frac{(2p-1)^2}{p^4} \right),
	\end{equation*}
	see Haukkanen and the second author \cite[Th.\ 6.4]{HauTot2018}. Compare formula \eqref{sum_gcud}
	to \eqref{gcd_a_b} and \eqref{asympt_f}.
	
	3. By \eqref{form_f} it is easy to see that given $m\in \N$, we have $f(a,b)=m$ if and only if
	$a=mki$, $b=mkj$, where $i,j,k,\ell\ge 1$ with $i+j=m\ell$, $\gcd(i,j)=1$, $\gcd(k,\ell)=1$.
	
	4. It would be interesting to study further asymptotic properties of the function $f$. For example, deduce asymptotic formulas for the sums $\sum_{a,b\le x} h(f(a,b))$, where $(\mu*h)(n)\ll n^{\beta}$ with $\beta \ge 2$, the case not covered by Theorem \ref{Th_h}. In particular, 
	deduce formulas for $\sum_{a,b\le x} (f(a,b))^t$, where $t\ge 2$.
	
	5. Consider the function 
	\begin{equation} \label{def_F}
		F(a,b)= \frac{\lcm(ab,a+b)}{\lcm(a,b)},
	\end{equation}
	which is the lcm analogue of the function $f$, defined by \eqref{def_f}. For every $a,b\in \N$ one has 
	\begin{equation} \label{form_F}
		F(a,b)= \frac{a+b}{f(a,b)} = \lcm(\delta,i+j),
	\end{equation}
	with the notation $a=\delta i$, $b=\delta j$, $\gcd(i,j)=1$. Hence $F$ is also integer-valued, and $\gcd(a,b)\mid F(a,b)$ for every $a,b\in \N$. Also,
	\begin{equation*}
		F(a,b)= \lcm(\lambda,a'+b'),
	\end{equation*}
	where $a=\lambda a'$, $b=\lambda b'$, $\lambda = \gcud(a,b)$. 
	
	We have the following result.
	
	\begin{theorem}
		\begin{equation*}
			\sum_{a,b\le x} F(a,b) = Lx^3 + O(x^2 (\log x)^3),
		\end{equation*}
		where the constant $L$ is given by \eqref{def_L}.
	\end{theorem}
	
	\begin{proof}
		According to \eqref{form_F},
		\begin{equation*}
			\sum_{a,b\le x} F(a,b) = \sum_{\substack{\delta i,\delta j\le x\\ \gcd(i,j)=1}} \lcm(\delta,i+j) =  \sum_{\substack{\delta i,\delta j\le x\\ \gcd(i,j)=1}} \frac{\delta(i+j)}{\gcd(\delta,i+j)} = 2 \sum_{\substack{\delta i,\delta j\le x\\ \gcd(i,j)=1}} \frac{\delta i}{\gcd(\delta,i+j)}. 
		\end{equation*}
		
		For the function $g(n)= \frac1{n} \prod_{p\mid n} (1-p)$ we have $\sum_{d\mid n} g(d)=\frac1{n}$ ($n\in \N$). We deduce that
		\begin{equation*}
			\sum_{a,b\le x} F(a,b) = 2 \sum_{\substack{\delta i,\delta j\le x\\ \gcd(i,j)=1}} \delta i \sum_{d\mid \gcd(\delta,i+j)} g(d) =
			2 \sum_{\substack{dsi,dsj \le x\\ \gcd(i,j)=1\\ d\mid (i+j)}} dsi\, g(d)
		\end{equation*}
		\begin{equation*}
			= 2 \sum_{d\le (2x)^{1/2}} dg(d) \sum_{s\le 2x/d^2} s \ \sum_{\substack{i,j \le x/(ds) \\ \gcd(i,j)=1\\ d\mid (i+j)}} i.
		\end{equation*}
		
		Here we need the following analogue of the formula in Lemma \ref{Lemma_V}:
		\begin{equation*}
			\sum_{\substack{i,j\le x\\ \gcd(i,j)=1\\ d\mid (i+j)}} i = \frac{x^3}{2\zeta(2)\psi(d)} + O\left(x^2\log x\right),
		\end{equation*}
		uniformly for $x>1$ and $d\in \N$. This follows from Lemmas \ref{Lemma_1} and \ref{Lemma_2} by partial summation.
		Now, the rest of the proof is similar to the proof of Theorem \ref{Th_h}.
	\end{proof}
	
	6. Consider the following $k$-variable generalizations of the functions $f$ and $F$, defined by \eqref{def_f} and \eqref{def_F}:
	\begin{equation*} 
		f(a_1,\ldots a_k)= \frac{\gcd(a_1\cdots a_k,a_1+\cdots +a_k)}{\gcd(a_1,\ldots,a_k)},
	\end{equation*}
	\begin{equation*} 
		F(a_1,\ldots a_k)= \frac{\lcm(a_1\cdots a_k,a_1+\cdots +a_k)}{\lcm(a_1,\ldots,a_k)},
	\end{equation*}
	which are also integer-valued functions.

	Steve Fan \\
	Department of Mathematics \\ 
	University of Georgia \\ 
	Athens, GA 30602, USA \\
	{\tt{Steve.Fan@uga.edu}}
	\\
	
	L\'aszl\'o T\'oth \\
	Department of Mathematics \\
	University of P\'ecs \\
	Ifj\'us\'ag \'utja 6, 7624 P\'ecs \\
	Hungary \\
	{\tt{ltoth@gamma.ttk.pte.hu}}

\end{document}